\documentclass[12pt]{amsart}
\usepackage{amsmath,amssymb,amsfonts,amsthm}
\usepackage{xcolor}
\usepackage{enumerate}
\usepackage[hyphens]{url}
\usepackage[utf8]{inputenc}
\usepackage{mathtools}   % loads »amsmath
\usepackage[normalem]{ulem}% use for \sout
\usepackage[color]{changebar}
\usepackage[margin=1in]{geometry}

\theoremstyle{definition} %Tira o itálico dos teoremas

\newtheorem{thm}{T}%[chapter]
\numberwithin{thm}{section}  %poderia ser subsection
\numberwithin{equation}{section}

\newtheorem{definition}[thm]{Definition}
\newtheorem{theorem}[thm]{Theorem}
\newtheorem{lemma}[thm]{Lemma}
\newtheorem{corollary}[thm]{Corollary}

\newtheorem{remark}[thm]{Remark}
\newtheorem{example}[thm]{Example}

\newcommand{\NN}{\mathbb{N}}
\newcommand{\Cst}{\mathrm{C}^*}

\newcommand{\Uu}{\mathcal{U}}

\begin{document}

\title{Equivalence of definitions of AF groupoid}

\author{Lisa Orloff Clark}
\author{Astrid an Huef} 
\author{Rafael P. Lima}
\author{Camila F. Sehnem}

\address[L.O.\ Clark, A. an\ Huef, and R.P.\ Lima]{School of Mathematics and Statistics, Victoria University of Wellington, PO Box 600, Wellington 6140, NEW ZEALAND}
\email{{lisa.clark@vuw.ac.nz}, {astrid.anhuef@vuw.ac.nz},
{rafael.lima@vuw.ac.nz}}

\address[C.F.\ Sehnem]{Department of Pure Mathematics, University of Waterloo,
200 University Avenue West,
Waterloo, Ontario N2L 3G1, CANADA}
\email{{camila.sehnem@uwaterloo.ca}}

\subjclass[2020]{22A22(primary), 46L05 (secondary)}
\keywords{elementary groupoid, AF groupoid}

\thanks{This research was funded by the Marsden Fund of the Royal Society of New Zealand (grant number 18-VUW-056)}

\begin{abstract}
We prove the equivalence of two definitions of AF groupoid in the literature: one by Renault and the other by Farsi, Kumjian, Pask and Sims. In both definitions, an AF groupoid is an increasing union of more basic groupoids, called elementary groupoids. Surprisingly, the two definitions of elementary groupoid are not equivalent; they  coincide if and only if the local homeomorphism that characterises them is a covering map.
\end{abstract}

\maketitle

\section{Introduction}

An approximately finite-dimensional $\Cst$-algebra, or AF algebra, is a $\Cst$-algebra that 
is a completion of an increasing union of finite-dimensional $\Cst$-algebras.  Elliot's famous classification theorem \cite[Theorem 4.3]{ELLIOTT} implies that AF algebras are completely classified by their K-theory. Given how well-behaved these $\Cst$-algebras are, it can be surprisingly difficult to decide if a particular $\Cst$-algebra is AF.  Kumjian, Pask and Raeburn \cite{KPR} show that the $\Cst$-algebra of a row-finite graph is AF if and only if the graph has no loops. However, determining graph conditions that ensure the $\Cst$-algebra of a higher-rank graph is AF is still an open problem, see \cite{ES2012}.

All AF algebras are $\Cst$-algebras of groupoids: given an AF algebra $A$, Renault shows in \cite[Proposition~III.1.15]{Renault} that there exists a particularly nice groupoid $G$ such that $A$ is isomorphic to the $\Cst$-algebra $\Cst(G)$ of $G$; he calls this groupoid an ``AF-groupoid''.
For Renault, $G$ is an increasing union of ``elementary groupoids'' which, locally, are Cartesian products of a totally  disconnected space with a countable transitive principal groupoid.  We note that it is possible for a non-AF groupoid to give rise to an AF $\Cst$-algebra, see for example \cite{MS2022}.

Over time, variations of this notion of ``AF groupoid'' began appearing in the literature. 
In 2003 \cite{Renault-AF}, Renault introduced AF equivalence relations and in 2004, Giordano, Putnam and Skau \cite{GPS} gave a different but equivalent definition. In 2012, Matui \cite{Matui} gave another definition of AF groupoid, which required the unit space to be compact. In 2019, Farsi, Kumjian, Pask and Sims \cite{FKPS} generalised this to groupoids with not necessarily compact unit spaces.  

In this paper, we focus on two notions of AF groupoid: Renault's original one and Farsi et al's.
In both, an AF groupoid is defined to be an increasing union of smaller groupoids with a more basic structure.  In each instance, these smaller groupoids are called ``elementary''.  We show that the two notions of ``elementary groupoid'' are not equivalent;   despite this,
%difference in the constructions of AF groupoids, 
unexpectedly, we show that the two corresponding definitions of AF groupoid are equivalent.

This paper is organized as follows. We begin with a section on preliminaries, and in Section~\ref{section:elementary} we give the details required in the definitions of elementary and AF groupoids both in the sense of Renault and of Farsi et al.
In Section \ref{section:nequivalence}, we compare the two notions of elementary groupoid. We prove in Theorem~\ref{thm:elementary2} that the definition of elementary groupoid of Farsi et al's  is more general than Renault's. Moreover, in Example~\ref{example:Cantorelementary2} we present a  groupoid that is elementary with respect to Farsi et al's definition but not Renault's.

In Theorem~\ref{thm:AFgpdequiv} we show that Farsi et al's and Renault's definitions of AF groupoid are equivalent. That every AF groupoid in Renault's definition is AF in Farsi et al's definition follows because every elementary groupoid in the sense of Renault is also an elementary groupoid according to Farsi et al. The other direction is more difficult: there we show that any sequence of elementary groupoids as in Farsi et al's definition can be written as an increasing union of a sequence of groupoids which are elementary in the sense of Renault. We finish this paper by showing in Corollary~\ref{cor:AF-rel-groupoid} that these notions of AF groupoid also coincide with the notion of AF-equivalence relation by Giordano, Putnam and Skau.

%\todo{Say somewhere what we know about the $\Cst$-algebras of $R(\sigma)$. For example, reference \cite{CaHR-Fell} and \cite{ArchSom} for $\Cst(R(\sigma))$.}

\section{Preliminaries}
\label{section:prelims}

Let $X$ and $Y$ be second-countable, locally compact and Hausdorff spaces, and let $\sigma:X\to Y$ be a continuous map. We say that $\sigma$ is a \emph{local homeomorphism} if for every $x\in X$ there exists an open neighbourhood $\Uu$ of $X$ such that the restriction $\sigma\vert_\Uu$ of $\sigma$ to $\Uu$ is injective and $\sigma(\Uu)$ is open in $Y$. 
We say that $\sigma$ is a \emph{covering map} if it is surjective, and for every $y\in Y$ there exist an open neighbourhood $W$ of $y$, $N\in\NN\setminus\{0\}\cup\{\infty\}$ and $N$ disjoint open subsets  $\Uu_1, \Uu_2, \dots$  of $X$ such that $\sigma^{-1}(W)=\bigcup_{i=1}^N \Uu_i$ and  $\sigma\vert_{\Uu_i}$  is a homeomorphism of $\Uu_i$ onto $W$ for all $i$. Note that every covering map is a local homeomorphism.
We will use the following:

\begin{lemma}(See also \cite[Theorem~ 4.10]{FKPS})\label{lem:section}
Let $X$ and $Y$ be second-countable, locally compact, Hausdorff, totally disconnected spaces, and let $\sigma:X\to Y$ be a surjective local  homeomorphism.  Then there exists a section $\varphi:Y\to X$ for $\sigma$ that is a local homeomorphism.
\end{lemma}
\begin{proof}
    Since $\sigma$ is a local homeomorphism and $X$ is second-countable, we can cover $X$ with a countable family of open sets $\mathcal{U}_1, \mathcal{U}_2, \dots$ such that the restriction $\sigma\vert_{\mathcal{U}_i}: \mathcal{U}_i \rightarrow \sigma(\mathcal{U}_i)$ is a homeomorphism for all $\mathcal{U}_i$. Without loss of generality, we assume that these sets are compact open and disjoint because $X$ is locally compact and totally disconnected. We let $M \in \NN \cup \lbrace \infty \rbrace$ be the number of sets $\mathcal{U}_i$.

    Let $V_1 = \mathcal{U}_1$, and define for each $i\geq 1$
    \[
        V_{i+1} = \mathcal{U}_{i+1} \setminus \bigcup_{j=1}^{i} \sigma^{-1}(\sigma(\mathcal{U}_j)).
    \]
    Note that the sets $V_i$s are compact open, disjoint, and cover $X$. By rearranging the sets $V_i$ if necessary, let $N \leq M$ be such that
    \[
        Y = \bigcup_{i=1}^N \sigma(V_i),
    \]
    where each $V_i \neq \emptyset$. Define the map $\varphi: Y \rightarrow X$ by $\varphi\vert_{\sigma(V_i)} = \sigma\vert_{V_i}^{-1}$ for all $i < N + 1$.  
    Then $\varphi$ is a local homeomorphism, with image
    \begin{align*}
        \varphi(Y)
        = \bigcup_{i=1}^N \varphi(\sigma(V_i))
        = \bigcup_{i=1}^N \sigma\vert_{V_i}^{-1}(\sigma(V_i)) 
        = \bigcup_{i=1}^N V_i %= X.
    \end{align*}
     Let $y \in Y$. Then there exists a unique $i$ such that $\varphi(y) \in V_i$. this implies that $\sigma(\varphi(y)) = \sigma(\sigma\vert_{V_i}^{-1}(y)) = y$. Therefore $\varphi$ is a section of $\sigma$.
    %So $\varphi$ is surjective.
    %
    % Let $x \in X$. Then there exists a unique $i$ such that $x \in V_i$. Then $\sigma(x) \in \sigma(V_i)$ and $\varphi(\sigma(x)) = x$ by the definition of $\varphi$. Therefore $\varphi$ is a section of $\sigma$.    
\end{proof}

%\begin{proof} The continuous section constructed in \cite[Theorem~4.10]{FKPS} is open and locally injective, and hence a local homeomorphism. 
%\end{proof}

For basic details about topological groupoids, see \cite{Williams-groupoids}.
Let $\sigma:X\to Y$ be a continuous and surjective map. Then the equivalence relation
\[
R(\sigma)\coloneqq \{(x_1, x_2)\in X\times X : \sigma(x_1)=\sigma(x_2)\},
\]
with the subspace topology from $X\times X$, is a second-countable, locally compact, Hausdorff,  principal groupoid. The range and source maps $r,s\colon R(\sigma)\to X$ are given by $r(x_1,x_2)=x_1$ and $s(x_1,x_2)=x_2$, where we have identified the unit space of $R(\sigma)$ with $X$; the groupoid operations on $R(\sigma)$ are given by $(x_1,x_2)(x_2,x_3)=(x_1,x_3)$ and $(x_1,x_2)^{-1}=(x_2,x_1)$.
By \cite[Lemma~4.2]{CaHR-Fell}, the groupoid $R(\sigma)$ is \'etale if and only if  $\sigma$ is a local homeomorphism.

Let $G$ and $H$ be groupoids. Then the \emph{product groupoid} $G\times H$ has composable pairs $(G\times H)^{(2)}=G^{(2)}\times H^{(2)}$, and  product and inverse given by 
\[
(g, h)(g', h) = (g g', h h')
\quad\text{and}\quad
(g,h)^{-1} = (g^{-1}, h^{-1})
\]
for $(g,g')\in G^{(2)}$ and $(h,h')\in H^{(2)}$.
If $G, H$ are topological groupoids, then $G\times H$ with the  product topology is again a topological groupoid. 

The \emph{disjoint union groupoid} $G \sqcup H$ of $G$ and $H$
is the set-theoretical disjoint union $G \sqcup H= (G \times \lbrace 1 \rbrace) \cup (H \times \lbrace 2 \rbrace)$  with composable pairs $(G \sqcup H)^{(2)}=(G^{(2)}\times \{1\})\cup (H^{(2)}\times\{2\})$, and product and inverse given by, for $i=1,2$,
\[(g,i)(h,i) = (gh,i) \quad\text{and}\quad 
(g,i)^{-1} = (g^{-1}, i) 
\]
If $G^{(0)} \cap H^{(0)} = \emptyset$, then we may  identify $G \sqcup H$ with $G \cup H$.
 If $G$, $H$ are topological groupoids, then  $G \sqcup H$ equipped with the  the disjoint union topology is again a topological groupoid. 

\subsection*{Notation}
For $n\in\NN\setminus\{0\}$, we write  $[n] = \lbrace 1, \hdots, n \rbrace$. We also use  $[\infty] = \lbrace 1, 2, \hdots \rbrace$.

\section{Elementary and AF groupoids}
\label{section:elementary}

We start by giving the definitions of `elementary' and `AF' groupoid from \cite{Renault} and \cite{FKPS}; here is Renault's \cite[Definition~III.1.1]{Renault}.

\begin{definition}\label{def:Renault}  (Renault)
Let $n \in \NN\setminus\{0\}\cup \{\infty\}$. A groupoid $G$ is  \emph{elementary  of type} $n$ if it is isomorphic to the product of a second-countable, locally compact, Hausdorff space and a transitive principal groupoid on a set of $n$ elements. 
A groupoid $G$ is  \emph{Renault-elementary} if it is the disjoint union of a sequence of elementary groupoids $G_i$ of type $n_i$.
A groupoid $G$ is  \emph{approximately finite} (\emph{AF}) if its unit space is totally disconnected and there is an increasing sequence of open elementary subgroupoids $G_n$, with the same unit space, such that $G = \bigcup_{n=1}^\infty G_n$.
\end{definition}

In our notation, 
 an elementary groupoid of type $n\in\NN\setminus\{0\}\cup\{\infty\}$ is isomorphic to $X \times [n]^2$ for some $X$. Then every Renault-elementary groupoid is isomorphic to a disjoint-union groupoid  $\bigsqcup_i X_i \times [n_i]^2$.
 We also identify the unit space $(X \times [n]^2)^{(0)}$ with $X \times [n]$. 
The following definitions are from \cite[Definition~4.9]{FKPS}.

\begin{definition}\label{def:FKPS} (Farsi--Kumjian--Pask--Sims)
Let $G$ be a  locally compact, Hausdorff \'etale groupoid with totally disconnected unit space. Then $G$ is \emph{FKPS-elementary} if there exists a
surjective local homeomorphism $\sigma: X \rightarrow Y$ between second-countable, locally compact, Hausdorff, totally disconnected spaces $X$ and $Y$ such that $G$ is isomorphic to  
\[
R(\sigma)= \{(x_1, x_2)\in X\times X : \sigma(x_1)=\sigma(x_2)\}.
\]
A groupoid is \emph{AF} if it is the increasing union of open FKPS-elementary subgroupoids with the same unit space
%\footnote{We think that \cite[Definition 4.9]{FKPS}   accidentally omits that the union above should be an \emph{increasing} union of groupoids with \emph{the same unit space}.  With these additions, \cite[Definition 4.9]{FKPS}  is indeed a generalisation of  \cite[Definition~2.2]{Matui} to groupoids with non-compact unit spaces, as is claimed.}
%.\footnote{\loc{I think we have added second-countable to both definitions. I am not sure how we want to deal with this.  I *think* that we only use second-countability to show that the unit space $X$ can be written as a countable union of compact sets.  Makes me wonder if the two notions of AF are equivalent for non second countable groupoids. } }
\end{definition}

\begin{remark} We have a few comments concerning Definitions \ref{def:Renault} and \ref{def:FKPS}.
\begin{enumerate}
    \item[(i)] Since our main goal is to compare different definitions of AF groupoids in the literature, we will consider in our main results Renault-elementary groupoids with totally disconnected unit spaces.

        \item[(ii)] We believe that in \cite[Definition 4.9]{FKPS} the authors accidentally omitted in the definition of an AF groupoid that the union should be an \emph{increasing} union of groupoids with \emph{the same unit space}. With these additions, \cite[Definition 4.9]{FKPS}  is indeed a generalisation of Matui's \cite[Definition~2.2]{Matui} to groupoids with non-compact unit spaces, as is claimed. We have included these modifications in Definition~\ref{def:FKPS}. Notice also that those additions are used in the main results of \cite{FKPS}.

        %\item[(iii)]\rpl{It follows from Giordano, Putnam and Skau's \cite[Proposition 3.2 (iii)]{GPS} that their notion of AF equivalence relation \cite[Definitions 3.1 and 3.7] {GPS} is an example of Renault's definition of AF groupoids.}
        
\end{enumerate}
    \end{remark}

We also recall the definition of AF-equivalence relation by Giordano, Putnam and Skau \cite[Definitions 3.1 and 3.7]{GPS}. We show in Corollary~\ref{cor:AF-rel-groupoid} that this is equivalent to the above notions of AF groupoid.

\begin{definition}(Giordano, Putnam, Skau) Let $R$ be an equivalence relation on $X$ equipped with a topology that makes it an \'etale groupoid. We say that $R$ is a \emph{compact \'etale equivalence relation} (CEER) if $R \setminus X$ is a compact subset of $R$. We say that $R$ is an \emph{AF-equivalence relation} if $R^{(0)}$ is locally compact, Hausdorff, second countable, totally disconnected, and if $R$ is the inductive limit of a sequence of CEERs.
\end{definition}

\section{Comparing notions of AF groupoid}\label{section:nequivalence}
Our first aim is to show that if $G$ is FKPS-elementary, then it is Renault-elementary if and only if the associated surjective local homeomorphism is a covering map.

\begin{lemma}\label{lem:elementary}
    Let $G$ be a  Renault-elementary groupoid with totally disconnected unit space. Then there exist a second-countable, locally compact, Hausdorff  space $Y$ and a covering map $\sigma: G^{(0)} \rightarrow Y$ such that $G$ is isomorphic to $R(\sigma)$.
\end{lemma}

\begin{proof}
    Since $G$ is Renault-elementary, there  exist $N\in\NN\cup\{\infty\}$ and  $N$ second-countable, locally compact, Hausdorff, spaces $X_1, X_2, \dots$, and  $N$ elements $n_1, n_2, \dots \in\NN\cup\{\infty\}$ such that
\begin{equation}\label{G elementary}
G \cong \bigsqcup_{i=1}^N X_i \times [n_i]^2.
\end{equation}
For the rest of the proof we assume that $G$ equals the right-hand-side of~\eqref{G elementary}, and hence that $G^{(0)}$ equals $\bigsqcup_{i=1}^N X_i \times [n_i]$,
where we identify $[n_i]$ with the diagonal $\{(j,j): 1\leq j\leq n_i\}$ in $[n_i]^2$.
Further, by replacing $X_i$ with $X_i \times \lbrace i \rbrace$ if necessary, we may assume that the $X_i$ are already disjoint. 
Also, since $G^{(0)}$ is totally disconnected, the $X_i$ are as well.

We set $Y\coloneqq\bigcup_{i=1}^N X_i$ and 
let $\sigma: G^{(0)} \rightarrow Y$ be given by $\sigma(x,j) = x$. 
Then $\sigma$ is surjective and continuous. To see that $\sigma$ is  a covering map,  fix $y\in Y$. Then $y\in X_k$ for a unique $k$, and $X_k$ is open in $Y$. Then 
\[
\sigma^{-1}(X_k)
= X_k \times [n_k]
= \bigcup_{j=1}^{n_k} X_k \times \lbrace j \rbrace
\]
is a disjoint union of open subsets $X_k \times \lbrace j \rbrace$ of $G^{(0)}$, and the restriction of $\sigma$ to each $X_k \times \lbrace j \rbrace$ is a homeomorphism onto $X_k$. Thus $\sigma$ is a covering map.  

Finally, a straightforward computation shows that $\mu: G \rightarrow R(\sigma)$ defined by $\mu(x,j,k) = (x,j,x,k)$ is a groupoid homomeorphism.
\end{proof}
%}

The next lemma is a technical tool that we will use to prove Theorem~\ref{thm:elementary2}.

% \sout{Given a continuous section $\varphi$ of a surjective local homeomorphism $\sigma: X \rightarrow Y$, we denote $X_\varphi = \varphi(Y)$}
% \footnote{\todo{Using  $\varphi(Y)$ instead of $X_\varphi$ makes the statement of the next lemma more self-contained. Is there a reason for introducing new notation for $\varphi(Y)$?}.}

\begin{lemma}
\label{lemma:covmap-sets}
Let $\sigma: X \rightarrow Y$ be a covering map between second-countable, locally compact, Hausdorff and totally disconnected spaces $X$ and  $Y$. Let $\varphi: Y \rightarrow X$ be a section of $\sigma$ that is a local homeomorphism, and let $y \in Y$. Then there exist a compact  open neighbourhood $W$ of $y$,  a positive $n \in \mathbb{N} \cup \lbrace \infty \rbrace$, and  $n$ disjoint compact open subsets $\mathcal{U}_1, \mathcal{U}_2, \dots$ of $X$ such that
%}
\begin{enumerate}[(i)]
\item $\sigma^{-1}(W) = \bigcup_{i=1}^n \mathcal{U}_i$,
\item $\sigma\vert_{\mathcal{U}_i}: \mathcal{U}_i \rightarrow W$ is a homeomorphism for all $i \in \NN$ such that $i\leq n$,
\item\label{it3:lemcovmap} $\varphi(W) = \sigma^{-1}(W) \cap \varphi(Y) = \mathcal{U}_1$.  
\end{enumerate}
\end{lemma}
Note:  We use item~\eqref{it3:lemcovmap} when we build a Renault-elementary groupoid from a covering map, in particular, the $\mathcal{U}_1$ is the locally compact Hausdorff space in the product.

\begin{proof} We write $X_\varphi =\varphi(Y)$ for convenience. In particular, $X_\varphi$ is open.

 Since $\sigma$ is a covering map on the totally disconnected and second-countable space $X$, there exist a compact open neighbourhood $W$ of $y$,  a nonzero $n \in \mathbb{N} \cup \lbrace \infty \rbrace$ and $n$ disjoint compact open $\mathcal{O}_1, \mathcal{O}_2, \dots$ of $X$ with
\begin{align}
\label{eqn:invsigmaWOi}
\sigma^{-1}(W) = \bigcup_{j=1}^n \mathcal{O}_j,
\end{align}
and such that $\sigma\vert_{\mathcal{O}_j}: \mathcal{O}_j \rightarrow W$ is a homeomorphism for all $j$.
Set \[
\Uu\coloneqq \sigma^{-1}(W)\cap X_\varphi.
\]
Notice that $\Uu$ is nonempty because $\varphi(y)\in \Uu$ and that it is open because  $X_\varphi$ is open  and $\sigma$ is continuous. 
We claim that $\sigma\vert_{\Uu}: \Uu \rightarrow W$ is a homeomorphism. In particular, it then follows that $\Uu$ is compact because $W$ is. Since $\sigma$ is open and $\Uu$ is open, $\sigma\vert_{\Uu}$ is open. Since $\varphi$ is a section for $\sigma$, it follows that $\sigma$ is injective on $X_\varphi$, and hence  $\sigma\vert_{\Uu}$ is injective. We have 
\[\sigma(\Uu)=\sigma\big( \sigma^{-1}(W)\cap X_\varphi  \big)
\subseteq \sigma ( \sigma^{-1}(W))\cap\sigma(X_\varphi)
= W \cap Y
= W.
\]
For the reverse inclusion, let $w\in W$. Then $w=\sigma(\varphi(w))$ implies that $\varphi(w)\in\sigma^{-1}(W)\cap X_\varphi$, and hence $w\in\sigma(\Uu)$. Thus $\sigma(\Uu)=W$, and $\sigma\vert_{\Uu}: \Uu \rightarrow W$ is a homeomorphism as claimed. 

Since $\Uu$ is compact, by reordering $\mathcal{O}_1, \mathcal{O}_2, \dots$ if necessary, there exists $m\leq n$ such that, for $i \in [n]$, $\mathcal{O}_i \cap X_\varphi \neq \emptyset$ if and only if $1\leq i \leq m$. In particular, 
\[\Uu = \bigcup_{j=1}^m \mathcal{O}_j \cap X_\varphi.
\]
Let $1\leq j\leq m$. Then
\begin{align*}
\mathcal{O}_j 
&= \sigma\vert_{\mathcal{O}_j}^{-1}(W) 
= \sigma\vert_{\mathcal{O}_j}^{-1}(\sigma(\Uu)) \\
&= \sigma\vert_{\mathcal{O}_j}^{-1} \circ \sigma \left( \bigcup_{l=1}^m \mathcal{O}_l \cap X_\varphi \right) \\
&= \bigcup_{l=1}^m \sigma\vert_{\mathcal{O}_j}^{-1} \circ \sigma( \mathcal{O}_l \cap X_\varphi),
\end{align*}
and this union is disjoint because $\sigma$ is injective on $X_\varphi$
and because the sets $\mathcal{O}_l$ are disjoint.
We then obtain the disjoint union
\begin{align}\label{eq-unionOj}
\bigcup_{j=1}^m \mathcal{O}_j
&= \bigcup_{j,l=1}^m \sigma\vert_{\mathcal{O}_j}^{-1} \circ \sigma(\mathcal{O}_l \cap X_\varphi) \notag
\\
&= \bigcup_{j=1}^m \bigcup_{k=-1}^{m-2} \sigma\vert_{\mathcal{O}_j}^{-1} \circ \sigma(\mathcal{O}_{(j + k \textrm{ mod } m)+ 1} \cap X_\varphi). 
\end{align}
 For $i =1, 2, \dots, m$, define 
\begin{align*}
\mathcal{U}_i = \bigcup_{j=1}^m  \sigma\vert_{\mathcal{O}_j}^{-1} \circ \sigma(\mathcal{O}_{(j + i - 2 \textrm{ mod } m)+ 1} \cap X_\varphi).
\end{align*}
Note that $\mathcal{U}_1=\Uu$ and that $\Uu_i\cap \Uu_j=\emptyset$ unless $i=j$. 
As shown above for $i=1$, each  $\sigma\vert_{\Uu_i}: \Uu_i \rightarrow W$ is a homeomorphism onto $W$.
Moreover, 
\begin{align*}
\bigcup_{i=1}^m \mathcal{U}_i
&=\bigcup_{i=1}^m\bigcup_{j=1}^m
      \sigma\vert_{\mathcal{O}_j}^{-1} \circ \sigma(\mathcal{O}_{(j + i - 2 \textrm{ mod } m)+ 1} \cap X_\varphi)\\
      % &=\bigcup_{j=1}^m\bigcup_{i=1}^m
      % \sigma\vert_{\mathcal{O}_j}^{-1} \circ \sigma(\mathcal{O}_{(j + i - 2 \textrm{ mod } m)+ 1} \cap X_\varphi)\\
      &=\bigcup_{j=1}^m\bigcup_{l=-1}^{m-2}
      \sigma\vert_{\mathcal{O}_j}^{-1} \circ \sigma(\mathcal{O}_{(j + l \textrm{ mod } m)+ 1} \cap X_\varphi)
      =\bigcup_{j=1}^m \mathcal{O}_j
\end{align*}
using \eqref{eq-unionOj}. 

Finally, if $n > m$,  then  for $i>m$ we define $\mathcal{U}_i = \mathcal{O}_i$. Then $\Uu_1, \Uu_2,\dots$ are disjoint compact open subsets such that properties (i)--(iii) hold.
\end{proof}

\begin{theorem}
\label{thm:elementary2}
Let $\sigma: X \rightarrow Y$ be a surjective local homeomorphism 
between second-countable, locally compact, Hausdorff, totally disconnected spaces. Then $R(\sigma)$ is Renault-elementary if and only if $\sigma$ is a covering map.
\end{theorem}

\begin{proof}
By Lemma~\ref{lem:section} there exists a section $\varphi: Y \rightarrow X$ of $\sigma$ that is a local homeomorphism. Set $X_\varphi = \varphi(Y)$.
First, suppose that $R(\sigma)$
is Renault-elementary. We will show that $\sigma$ is a covering map. By Lemma~\ref{lem:elementary}, there  exists a space $\widetilde Y$, a covering map $\tilde\sigma:X\to \widetilde Y$ and an isomorphism $\mu:R(\sigma)\to R(\tilde\sigma)$ of topological groupoids. 

Since $\mu$ preserves the source and range and restricts to a homeomorphism $\mu|:X\to X$, for all $x,x'\in X$ we have 
\[
\sigma(x)=\sigma(x')\quad\Longleftrightarrow\quad \tilde\sigma(\mu|(x))=\tilde\sigma(\mu|(x')).
\]
It follows that there exists a well-defined quotient map $\rho:\widetilde Y\to Y$ such that $\sigma=\rho\circ\tilde\sigma\circ\mu|$. 

We claim that $\rho$ is a homeomorphism. Suppose that $\rho(y)=\rho(y')$. Choose $x,x'\in X$ such that $y=\sigma(x)$ and $y'=\sigma(x')$.
Then
\[
\tilde\sigma(\mu|(x))=\rho(\sigma(x))=\rho(\sigma(x'))=\tilde\sigma(\mu|(x'))
\]
which implies $\sigma(x)=\sigma(x')$, that is, $y=y'$. Thus $\rho$ is injective.  To see that $\rho$ is open, let $U$ be open in $\widetilde Y$. Then $\rho(U)=\sigma(\mu|(\tilde\sigma^{-1}(U)))$ is open because $\sigma$ and $\mu$ are, and $\tilde \sigma$ is continuous. Thus $\rho$ is a homeomorphism as claimed.   Now $\sigma=\rho\circ\tilde\sigma\circ\mu|$ is a covering map because $\tilde\sigma$ is and $\mu|$ and $\rho$ are homeomorphisms. 

Conversely,  suppose that $\sigma$ is a covering map and  let $y \in Y$. By Lemma \ref{lemma:covmap-sets}, there exist a compact open neighbourhood $\widetilde{W_y}$ of $y$, a positive $n_y \in \NN \cup \{\infty\}$, and a sequence of $n_y$ disjoint compact open sets $\widetilde{\mathcal{U}}_1^{(y)}, \widetilde{\mathcal{U}}_2^{(y)}, \dots$ such that
\begin{align*}
\sigma^{-1}(\widetilde{W}_y)
= \bigcup_{j=1}^{n_y} \widetilde{\mathcal{U}}_j^{(y)},
\end{align*}
each $\sigma\vert_{\widetilde{\mathcal{U}}_j^{(y)}}: \widetilde{\mathcal{U}}_j^{(y)} \rightarrow \widetilde{W}_y$ is a homeomorphism and $\sigma^{-1}(\widetilde{W}_y) \cap X_\varphi = \widetilde{\mathcal{U}}_1^{(y)}$.

Since $Y$ is Hausdorff and second countable, there exists a possibly finite sequence of $N$ elements $y_1, y_2, \dots$ in $Y$ such that $\widetilde{W}_{y_1}, \widetilde{W}_{y_2}, \dots$ cover $Y$ and, for all $i$, we have
\begin{align}
\label{eqn:Wtildeyneeded2}
\widetilde{W}_{y_i}
\not\subset
\bigcup_{\substack{k=1 \\ k \neq i}}^N \widetilde{W}_{y_k}.
\end{align}
We set $W_1 \coloneqq \widetilde{W}_{y_1}$, and $W_i \coloneqq \widetilde{W}_{y_i} \setminus \bigcup_{k=1}^{j-1} W_k$ for $j \geq 2$. Property~\eqref{eqn:Wtildeyneeded2} guarantees that all the  $W_j$ are nonempty. Moreover, the union $Y = \bigcup_{i=1}^N W_i$ is disjoint.

For each $i \in [N]$, let $n_i = n_{y_i}$. Given $j \in [n_i]$, set $\mathcal{U}_j^{(i)} \coloneqq \widetilde{\mathcal{U}}_j^{(y_i)} \cap \sigma^{-1}(W_i)$. Then
\begin{enumerate}[(i)]
\item $\sigma\vert_{\mathcal{U}_j^{(i)}}: \mathcal{U}_j^{(i)} \rightarrow W_i$ is a homeomorphism,

\item the union $\sigma^{-1}(W_i) = \bigcup_{j=1}^{n_i} \mathcal{U}_j^{(i)}$ is disjoint, and
\item $\sigma^{-1}(W_i) \cap X_\varphi = \mathcal{U}_1^{(i)}$.
\end{enumerate}
So for each $i$, $\varphi(W_i)=\mathcal{U}_1^{(i)}$.
Further, $X_\varphi$ is the disjoint union $X_\varphi = \bigcup_{i=1}^N \mathcal{U}_1^{(i)}$. 

Now consider the disjoint union 
\begin{align*}
G \coloneqq \bigcup_{i=1}^N \varphi(W_i) \times [n_i]^2  = \bigcup_{i=1}^N \mathcal{U}_1^{(i)} \times [n_i]^2.
\end{align*}
Then $G$ is a Renault-elementary groupoid with totally disconnected unit space.
We claim that $R(\sigma)$ is isomorphic to $G$.

 Let $(u,v)\in R(\sigma)$, so that $\sigma(u)=\sigma(v)$. There exists a unique $i\in[N]$ such that $\sigma(u)=\sigma(v)\in W_i$. Since $\sigma^{-1}(W_i)$ is equal to the disjoint union $\bigcup_{j=1}^{n_i} \mathcal{U}_j^{(i)}$, there exist unique $j,k\in [n_i]$ such that 
 $u\in \mathcal{U}_j^{(i)}$ and $v \in \mathcal{U}_k^{(i)}$.
 Thus $\varphi(\sigma(u)) \in \mathcal{U}_1^{(i)}$ and $\rho(u,v)\coloneqq(\varphi(\sigma(u)), j, k)$ gives a well-defined map $\rho: R(\sigma) \rightarrow G$. 
We claim that $\rho$ is an isomorphism of topological groupoids.

To see that $\rho$ is a homomorphism, let $u, v, w \in X$  such that $\sigma(u) = \sigma(v) = \sigma(w)$. Let $i$ be the unique element of  $[N]$ such that $\sigma(u)\in W_i$. Then $\rho(u,v)=(\varphi(\sigma(u)), j, k)$ and $\rho(v,w)=(\varphi(\sigma(u)), k, l)$ for  unique $j,k,l\in [n_i]$.  Thus $\rho(u,v)$ and $\rho(v,w)$ are composable, and 
\[
\rho(u,v)\rho(v,w)=(\varphi(\sigma(u)), j, l)=\rho(u,w)=\rho\big((u,v)(v,w)  \big).
\]
Further,
\[
\rho\big((u,v)^{-1}\big)=\rho(v,u)=(\varphi(\sigma(v)), k, j)=(\varphi(\sigma(v)), j, k)^{-1}=\rho(u,v)^{-1}.
\]
Thus $\rho$ is a homomorphism.

To see that $\rho$ is injective, let $(u_1,v_1), (u_2, v_2)\in R(\sigma)$, and suppose that $\rho(u_1, v_1) = \rho(u_2, v_2)$. Then there exists a unique $i \in [N]$ such that $\sigma(u_1) \in W_i$, and there are unique $j,k\in [n_i]$ such that $\rho(u_1, v_1) = (\varphi(\sigma(u_1)), j, k)$. By assumption, 
\[(\varphi(\sigma(u_1)), j, k)=\rho(u_1, v_1) = \rho(u_2, v_2)= (\varphi(\sigma(u_2)), j, k).\]
Thus $\varphi(\sigma(u_1))=\varphi(\sigma(u_2))$ and hence $\sigma(u_1)=\sigma(\varphi(\sigma(u_1)))=\sigma(\varphi(\sigma(u_2))=\sigma(u_2)$.  Now $u_1, u_2\in \Uu^{(i)}_j$ and $v_1, v_2\in \Uu^{(i)}_k$ with $\sigma(u_1)=\sigma(u_2)=\sigma(v_1)=\sigma(v_2)$.  Since $\sigma$ is injective on $\mathcal{U}_j^{(i)}$ and $\mathcal{U}_k^{(i)}$,  $u_1 = u_2$ and $v_1 = v_2$. Thus $\rho$ is injective.

To see that $\rho$ is surjective, let $(x,j,k) \in G$.  Since $x\in X_\varphi$, there exists a unique $i\in [N]$ with $x\in \Uu_1^{(i)}$, and then $j,k\in [n_i]$.
Let $u = \sigma\vert_{\mathcal{U}_j^{(i)}}^{-1}(\sigma(x))$ and $v = \sigma\vert_{\mathcal{U}_k^{(i)}}^{-1}(\sigma(x))$. Then $\sigma(u) = \sigma(v) \in W_i$ and 
\[
\varphi(\sigma(u))=\varphi(\sigma(\sigma\vert_{\mathcal{U}_j^{(i)}}^{-1}(\sigma(x))))=\varphi(\sigma(x))=x
\]
because $x\in X_\varphi$. 
Hence $(u,v) \in R(\sigma)$ with $\rho(u,v) = (x,j,k)$.  Thus $\rho$ is surjective.  Moreover, we now have a formula for $\rho^{-1}\colon G\to R(\sigma)$ given by 
\begin{equation}\label{eq: rho inverse}
\rho^{-1}(x, j, k)=\big( \sigma\vert_{\mathcal{U}_j^{(i)}}^{-1}(\sigma(x)), \sigma\vert_{\mathcal{U}_k^{(i)}}^{-1}(\sigma(x))\big)
\end{equation}
where $i\in [N]$ is unique such that $x\in \Uu^{(i)}_1$.

To see that $\rho$ is continuous, suppose that $(u_n, v_n)\to (u,v)$ in $R(\sigma)$ as $n\to\infty$. Thus for all $n$, $\sigma(u_n)=\sigma(v_n)$, and there exist unique $i \in [N]$ and $j,k \in [n_i]$ such that $\sigma(u)=\sigma(v)\in W_i$ and $\rho(u,v)=(\varphi(\sigma(u)), j,k)$. Since $u_n\to u$ and $v_n\to v$ in $X$ we eventually have $\sigma(u_n)=\sigma(v_n)\in W_i$,  $u_n\in \Uu_j^{(i)}$ and $v_n\in \Uu_k^{(i)}$. Thus  $\rho(u_n, v_n) = (\varphi(\sigma(u_n)), j, k)$ eventually. Since $\varphi(\sigma(u_n))\to \varphi(\sigma(u))$ it follows that $\rho(u_n, v_n)\to 
\rho(u,v)$. Thus $\rho$ is continuous. 

To see that $\rho^{-1}$ is continuous, suppose that $(x_n, j_n, k_n)\to (x, j, k)$ in $G$. Say $x\in \Uu^{(i)}_1$. Then $x_n\to x$ in $X$ and $x_n\in \Uu^{(i)}_1$ eventually, and also $j_n=j, k_n=k\in [n_i]$ eventually. From~\eqref{eq: rho inverse} we have
$\rho^{-1}(x_n, j, k)=\big( \sigma\vert_{\mathcal{U}_j^{(i)}}^{-1}(\sigma(x_n)), \sigma\vert_{\mathcal{U}_k^{(i)}}^{-1}(\sigma(x_n))\big)$, and hence
$\rho^{-1}(x_n, j_n, k_n)\to \rho^{-1}(x, j, k)$ by continuity of $\sigma, \sigma\vert_{\mathcal{U}_j^{(i)}}^{-1}$ and $\sigma\vert_{\mathcal{U}_k^{(i)}}^{-1}$.
We conclude that $\rho$ is an isomorphism of topological groupoids, and it follows that $R(\sigma)$ is Renault-elementary.
\end{proof}

Every covering map is a surjective local homeomorphism, and every surjective local homeomorphism on a compact space is a covering map.  Thus we obtain the following corollary of Theorem~\ref{thm:elementary2}.

\begin{corollary} 
\label{cor}Let $G$ be a second-countable, locally compact, Hausdorff groupoid with a compact and totally disconnected unit space. Then $G$ is Renault-elementary if and only if $G$ is FKPS-elementary.
\end{corollary}

The following example gives a surjective local homeomorphism $\sigma$ that is not a covering map.  Thus it follows from Theorem~\ref{thm:elementary2} that $R(\sigma)$ is FKPS-elementary but not Renault-elementary.

\begin{example}  
\label{example:Cantorelementary2} 
Let $C \subset [0,1]$ be the Cantor set, and write $2\pi C$ for its scaled copy in $[0, 2\pi]$. Let
\begin{align*}
X \coloneqq\big(2\pi C \cup (2\pi C+ 2\pi)\big) \setminus \lbrace0, 4\pi \rbrace
\subset (0, 4\pi).
\end{align*}
Let $\theta: \mathbb{R} \rightarrow S^1$ be the map $\theta(t) = (\cos t, \sin t)$. Then $\theta$ is a surjective local homeomorphism (see \cite[Theorem 53.1]{Munkres}).   Then the restriction 
$\sigma: X \rightarrow \theta(X)$ of $\theta$ to $X$, where $\theta(X)$ is equipped with the subspace topology from $S^1$,  is a surjective local homeomorphism. 

Suppose, looking for a contradiction,  that $\sigma$ is a covering map. Let $y = (1,0)$. Then $y$ has a compact open neighbourhood $W$ such that $\sigma^{-1}(W)$ is the disjoint union on a possibly infinite family of $n$ subsets $\mathcal{U}_1, \mathcal{U}_2, \dots$ where $\sigma\vert_{\mathcal{U}_i}: \mathcal{U}_i \rightarrow W$ is a homeomorphism for all $i$. Then $n$ is the cardinality of $\sigma^{-1}(z)$ for all $z \in W$. 

Since $\theta$ is continuous and since $X$ has no isolated points, there exists an $x \in \sigma^{-1}(W)$ such that $\sigma(x) \neq (1,0)$. Assume that $x > 2\pi$, as the proof for $x < 2\pi$ is analogous. The definition  of $X$ implies that $x - 2\pi \in X$, and the periodicity of $\theta$ implies that $\sigma(x - 2\pi) = \theta(x)$. Then $n = \vert \sigma^{-1}(\sigma(x)) \vert \geq 2$. However, $\sigma^{-1}((1,0)) = \lbrace 2\pi \rbrace$, which implies that $n = 1$, a contradiction. Therefore $\sigma$ is not a covering map.
\end{example}

Although the two definitions of elementary groupoid are not equivalent, we now prove that the two definitions of AF groupoid are equivalent.

\begin{theorem}
\label{thm:AFgpdequiv}
A groupoid is AF in the sense of Renault (Definition~\ref{def:Renault}) if and only if it is AF in the sense of FKPS (Definition~\ref{def:FKPS}).
\end{theorem}
%\aah{
\begin{proof} 
Let $G$ be AF according to Definition~\ref{def:Renault}. Then $G = \bigcup_n G_n$ is the increasing union of a sequence of Renault-elementary open subgroupoids $G_n$ with the same totally disconnected unit space $X$. By Lemma~\ref{lem:elementary}, for every $n$, there exists a covering map $\sigma_n: X \rightarrow Y_n$  such that $G_n \cong R(\sigma_n)$. Thus   $G$ is  AF according to Definition~\ref{def:FKPS}.

Conversely, let $G$ be an AF groupoid according to Definition~\ref{def:FKPS}.  Then $G = \bigcup_n G_n$ is an increasing union of open FKPS-elementary subgroupoids $G_n$ with the same unit space $X=G^{(0)}$. Then for all $n$ there exist second-countable, locally compact, Hausdorff, totally disconnected spaces $V_n, Y_n$ and a surjective local homeomorphism $\sigma_n : V_n \rightarrow Y_n$ such that $G_n \cong R(\sigma_n)$; since this isomorphism restricts to a homeomorphism of unit spaces, we may  assume that $V_n=X$ for each $n$.

If $X$ is compact, the result follows from Corollary~\ref{cor}.  So we assume $X$ is not compact. Since $X$ is locally compact, totally disconnected and second countable, there exists an infinite increasing sequence of compact open subsets $X_1, X_2, \dots$ such that $X = \bigcup_{k=1}^\infty X_k$.  We may, by replacing $Y_n$ with $Y_n\times\{1\}$ if necessary, assume that $X$ and $Y_n$ are disjoint for all $n$.
Let $n, k \geq 1$. Define   $\sigma_{n,k}: X \rightarrow \sigma_n(X_k) \cup (X \setminus X_k)\subset Y_n\sqcup X$ by
\[
\sigma_{n,k}(x)
= \begin{cases}
\sigma_n(x) & \text{if } x \in X_k, \\
x & \text{otherwise.}
\end{cases}
\]
Since $\sigma_{n,k}\vert_{X_k}: X_k \rightarrow \sigma_n(X_k)$ is the restriction of the surjective local homeomorphism $\sigma_n$ to the compact open set $X_k$, we deduce that $\sigma_{n,k}\vert_{X_k}$ is a covering map. Further, the restriction of $\sigma_{n,k}$ to $X \setminus X_k$ is the identity map.  Since $X_k$ and $X\setminus X_k$ are open, it follows  that $\sigma_{n,k}$ is a covering map.

For the next few claims we make the following observation about $(x,x') \in R(\sigma_{n,k})$. If 
 $x \in X_k$, then $\sigma_{n,k}(x') = \sigma_{n,k}(x) \in \sigma_n(X_k)$, which implies that $x' \in X_k$ because $Y_n$ and $X$ are disjoint.  If $x\notin X_k$, then  $x=\sigma_{n,k}(x')\in X_k$ implies that $x'\notin X_k$  because $Y_n$ and $X$ are disjoint, and then $x=x'$.

We claim that $G_n$ is the increasing union of the $R(\sigma_{n,k})$, that is,
\begin{equation}\label{eq-containment1}
R(\sigma_{n,k}) \subset R(\sigma_{n,k+1})\subset \bigcup_{k=1}^\infty R(\sigma_{n,k})=G_n.
\end{equation}
Let $(x,x') \in R(\sigma_{n,k})$. First assume that $x \in X_k$. Then $x'\in X_k$ also.  Then
$\sigma_n(x) = \sigma_{n,k}(x) = \sigma_{n,k}(x') = \sigma_n(x')$, and hence $(x,x')\in  G_n$.
Moreover, $x, x' \in X_k \subset X_{k+1}$ implies that
$\sigma_{n,k+1}(x) = \sigma_{n, k+1}(x')$, and hence $(x,x')\in R(\sigma_{n,k+1})$. 
Second, assume that $x \notin X_k$. Then  $x=x'$,  and hence $(x,x')$ is an element of both $ R(\sigma_{n,k+1})$ and $G_n$. Thus 
$R(\sigma_{n,k}) \subset R(\sigma_{n,k+1})\subset G_n$.
Finally, if $(x,x')\in G_n$ then there exists $k$ such that $x,x'\in X_k$, and then $(x,x')\in R(\sigma_{n,k})$, proving the claim.

Next, we claim that
\begin{equation}\label{eq-containment2}
R(\sigma_{n,k}) \subset R(\sigma_{n+1, k}),
\end{equation}
To see this, let $(x,x') \in R(\sigma_{n,k})$. Then $(x,x')\in G_n \subset G_{n+1}$, and hence $\sigma_{n+1}(x)=\sigma_{n+1}(x')$. If $x_n\in X_k$, then $x' \in X_k$, and 
$\sigma_{n+1, k}(x) = \sigma_{n+1}(x) = \sigma_{n+1}(x') = \sigma_{n+1, k}(x')$, giving $(x,x')\in R(\sigma_{n+1, k})$.
If $x\notin X_k$, then $x'\notin X_k$, and $x=\sigma_{n,k}(x)=\sigma_{n,k}(x')=x'$, giving $(x,x')\in R(\sigma_{n+1, k})$.

To see that $R(\sigma_{n,k})$ is open in $G$, let $(x,x') \in R(\sigma_{n,k})$.  First, suppose that  $x \in X_k$. Then $x' \in X_k$. Then $G_n\cap (X_k\times X_k)$ is open in $G$ because $G_n$ is, and for $(u,v) \in G_n\cap (X_k\times X_k)$, we have 
$\sigma_{n,k}(u)
= \sigma_n(u)
= \sigma_n(v)
= \sigma_{n,k}(v),
$
giving $(u,v) \in R(\sigma_{n,k})$.  Second, suppose that  $x \notin X_k$. Then $x' \notin X_k$ and $x=\sigma_{n,k}(x)=\sigma_{n,k}(x')=x'$. Let $\Uu$ be an open neighbourhood of $x$ contained in $X\setminus X_k$ such that $\sigma_n\vert_\Uu$ is injective. 
Then $G_n\cap (\Uu\times \Uu)$ is open in $G$  and for $(u,v) \in G_n\cap (\Uu\times \Uu)$, we have  $u=v$, again
giving $(u,v) \in R(\sigma_{n,k})$.  Thus $R(\sigma_{n,k})$ is open in $G$. 
 
Using \eqref{eq-containment1} and \eqref{eq-containment2} we have
\[
G = \bigcup_{n=1}^\infty G_n = \bigcup_{n=1}^\infty \bigcup_{k=1}^\infty R(\sigma_{n,k})= \bigcup_{k=1}^\infty R(\sigma_{k,k}),
\]
 an increasing union of open subgroupoids $R(\sigma_{k,k})$ with unit space $X$.
Since each $\sigma_{k,k}$ is a covering map, it follows from  Theorem \ref{thm:elementary2} that $R(\sigma_{k,k})$ is elementary. Therefore $G$ is AF according to Definition~\ref{def:Renault}.
\end{proof}
%}

\begin{corollary}\label{cor:AF-rel-groupoid}
A groupoid is AF if and only if it is an AF-equivalence relation.
\end{corollary}
\begin{proof}
Let $R$ be an AF-equivalence relation. Then $R = \bigcup_{n=0}^\infty R_n$ is the inductive limit of a sequence of CEERs on a locally compact, Hausdorff, second countable, totally disconnected space $X$. It follows from \cite[Lemma 3.4]{GPS} that $R_n$ is Renault-elementary. Then $R$ is AF.

Conversely, let $G$ be an AF groupoid with unit space $X$. Then $G$ is the inductive limit of sequence of groupoids $G_n$, and there exists a surjective local homeomorphism $\sigma_n: X \rightarrow Y_n$ such that $G_n \cong R(\sigma_n)$. We can assume, without loss of generality, that this sequence of groupoids is infinite. Write $X = \bigcup_{n=1}^\infty X_n$ as an increasing union of compact open subsets. Then
\begin{equation*}
G 
= \bigcup_{n=1}^\infty G_n
= \bigcup_{n=1}^\infty G_n \vert_{X_n}
= \bigcup_{n=1}^\infty (G_n \vert_{X_n} \sqcup (X \setminus X_n)).
\end{equation*}
Note that the right-hand side of the equation above is an increasing union of \'etale equivalence relations. Furthermore,
\begin{equation*}
(G_n \vert_{X_n} \sqcup (X \setminus X_n)) \setminus X_n = G_n\vert_{X_n} \cong R(\sigma_n)\vert_{X_n}
\end{equation*}
is compact. Therefore each $(G_n \vert_{X_n} \sqcup (X \setminus X_n))$ is a CEER, and $G$ is an  AF-equivalence relation.
\end{proof}

\bibliographystyle{plain}
\bibliography{equivalence-AF-groupoid}{}
\end{document}